\documentclass[12pt,a4paper]{amsart}
\usepackage{amsmath,graphicx,amssymb}
\usepackage{amsfonts,color}
\usepackage{amscd}
\usepackage[latin2]{inputenc}
\usepackage{t1enc}
\usepackage[mathscr]{eucal}
\usepackage{indentfirst}
\usepackage{pict2e}
\usepackage{epic}
\usepackage{url}
\usepackage{epstopdf}
\usepackage{listings}
\usepackage{subcaption}

\numberwithin{equation}{section}
\usepackage[margin=2.6cm]{geometry}

\usepackage{tikz}

\usetikzlibrary{calc,decorations.pathreplacing}
\usetikzlibrary{quotes,angles}

\tikzset{point/.style={circle,fill,draw,inner sep=0,minimum size=3pt}}
\tikzset{vertex/.style={circle,fill,draw,inner sep=0,minimum size=7pt}}
\tikzset{overtex/.style={circle,fill=none,draw,inner sep=0,minimum size=7pt}}

\theoremstyle{plain}
\newtheorem{theorem}{Theorem}[section]
\newtheorem{lemma}[theorem]{Lemma}

\newtheorem{corollary}[theorem]{Corollary}

\newtheorem{definition}[theorem]{Definition}
\newtheorem*{definition*}{Definition}

\newtheorem{proposition}[theorem]{Proposition}
\newtheorem{question}[theorem]{Question}
\theoremstyle{remark}
\newtheorem{remark}[theorem]{Remark}

\begin{document}

\title[Some applications of Wagner's subgraph polynomial]{Some applications of Wagner's weighted subgraph counting polynomial}

\author{Ferenc Bencs}\thanks{Supported by the NKFIH (National Research, Development and Innovation Office, Hungary) grant KKP-133921.}
\address{\small Alfr\'ed R\'enyi Institute of Mathematics\\
\small Budapest, Hungary\\
\small\tt ferenc.bencs@gmail.com}

\author{P\'eter Csikv\'ari}\thanks{Supported by the Counting in Sparse Graphs Lend\"ulet Research Group.}
\address{\small Alfr\'ed R\'enyi Institute of Mathematics and \\
\small Department of Computer Science\\
\small E\"{o}tv\"{o}s Lor\'{a}nd University\\
\small Budapest, Hungary\\
\small\tt peter.csikvari@gmail.com}

\author{Guus Regts}\thanks{Supported by a NWO Vidi grant, VI.Vidi.193.068}
\address{\small Korteweg de Vries Institute for Mathematics\\
\small University of Amsterdam\\
\small Amsterdam, The Netherlands\\
\small\tt guusregts@gmail.com}


\begin{abstract}
We use Wagner's weighted subgraph counting polynomial to prove that the partition function of the anti-ferromagnetic Ising model on line graphs is real rooted and to prove that roots of the edge cover polynomial have absolute value at most $4$. 
We more generally show that roots of the edge cover polynomial of a $k$-uniform hypergraph have absolute value at most $2^k$, and discuss applications of this to the roots of domination polynomials of graphs.
We moreover discuss how our results relate to efficient algorithms for approximately computing evaluations of these polynomials.
\end{abstract}



\maketitle

\section{Introduction}
The investigation of the location of zeros of different partition functions of graphs and hypergraphs is a topic gaining more and more interest. 
The reason for this is that these partition functions are related to several topics such as statistical physics, combinatorics and computer science.
In statistical physics absence of complex zeros near the real axis implies absence of phase transition (in the Lee-Yang sense~\cite{yang1952statistical}).
In computer science it is related to the design of efficient approximation algorithms for computing evaluations of partition functions and graph polynomials. A recent approach by Barvinok~\cite{barvinok2016} combined with results from~\cite{patel2017} shows that zero-free regions for graph polynomials imply fast (polynomial time) algorithms for approximating evaluations when restricted to bounded degree graphs. 

In this note we give two new zero-free regions, one for the anti-ferromagnetic Ising model on line graphs, and one for the edge cover polynomial. 
For both these polynomials efficient approximation algorithms were known on the positive real line. 
For the Ising model this was based on the Monte Carlo Markov chain approach~\cite{dyer2020polynomial}, and for the edge cover polynomial on correlation decay~\cite{lin2014simple,liu2014fptas} as well as on the Monte Carlo Markov chain approach~\cite{bezakova2009sampling,huang2016canonical}.
Our results yield new efficient algorithms for these polynomials on bounded degree graphs, not only for evaluations on the positive real line, but also for complex evaluations.
More importantly they further stress the connection between absence of zeros and the existence of efficient algorithms.

Both our results are based on two short applications of a general technique of Wagner~\cite{wagner2009}.

\subsection*{The Ising model on line graphs}
Our first example is the Ising model of line graphs.
Let $G = (V,E)$ denote a simple graph and let $z,  b  \in \mathbb{C}$. The \emph{partition function of the Ising model} $Z_G(z,  b )$ is defined as
\[
    Z_G(z) = Z_G(z,  b ) = 
    \sum_{U \subseteq V} z^{\lvert U \rvert} \cdot b ^{\lvert \delta(U) \rvert},
\]
where $\delta(U)$ denotes the set of edges with one endpoint in $U$ and one endpoint
in $U \setminus V$. In this paper, we fix $b >0$ and consider the partition function $Z_G(z)$ as a polynomial in $z$. 
The case $ b  < 1$ is often referred to as the \emph{ferromagnetic case},
while $ b  > 1$ is referred to as the \emph{anti-ferromagnetic} case.
	
In this paper, we will investigate the anti-ferromagnetic Ising model for the class of line graphs. 
The line graph of a graph $G$ is a graph $L(G)$ with vertices being the edges of $G$ and two edges being connected if they share a common vertex. In particular, in Section~\ref{sec:Ising} we prove the following theorems.
\begin{theorem}\label{th:main}
Let $G$ be a graph and let $b\ge 1$. Every root of $Z_{L(G)}(z,b)$ is real and negative.
\end{theorem}

We can also deduce some information on the location of the so-called Fisher zeros when we put a bound on the maximum degree.

\begin{theorem}\label{th:fisher}
 For any $\Delta\ge 2$ and $0\le\alpha<\pi/2$, there exists an open set $U_{\Delta,\alpha} \subseteq \mathbb{C}$ containing the interval $[1,\infty)$, such that if $G$ has maximum degree at most $\Delta$, then for $b\in U_{\Delta,\alpha}$ and $\lambda\in\{z\in\mathbb{C}~|~|\arg(z)|<\pi-2\alpha\}$ we have
 \[
    Z_{L(G)}(\lambda,b)\neq 0.
 \]
 In particular, $Z_{L(G)}(1,b)\neq 0$ for all $b\in U_{\Delta,\alpha}$.
\end{theorem}

A way to interpret the previous theorem is that there is no phase transition of the anti-ferromagnetic Ising-model on line graphs in the Fisher~\cite{fisher} sense. This phenomenon was already observed by Sy{\^o}zi in \cite{syozi1951statistics} for the Kagom\'e-lattice, which is the line graph of the hexagonal lattice. 
Recently, a variant of the absence of phase transitions for line graphs was also proven by Dyer, Heinrich, Jerrum, and M{\"u}ller~\cite{dyer2020polynomial}. 
In particular, they proved that for any choice of $b$ and $\lambda$, there exists a fully polynomial time randomized algorithm to approximate $Z_{L(G)}(\lambda,b)$. 
Using Theorem~\ref{th:fisher}, one can show that with a uniform bound on the maximum degree, we can obtain a deterministic algorithm for the same task.
Let us briefly explain.
A successful approach for obtaining an approximation algorithm was proposed by Barvinok \cite{barvinok2016}, based on truncating the Taylor series of the logarithm of partition functions over a connected zero-free domain. 
In \cite{patel2017}, this method was improved so as to run in polynomial time on bounded degree graphs.
By combining this approach (see also~\cite{liu2019ising}) with the previous corollary and the existence of a zero-free disk around zero from \cite[Remark 24]{peters2020location}, we obtain the following corollary. 

\begin{corollary}\label{cor:FPTAS}
 For   $\Delta\ge 2$ and $0\le \alpha<\pi/2$, let $U_{\Delta,\alpha}$ given by Theorem~\ref{th:fisher}. Let $b\in U_{\Delta,\alpha}$ and $\xi\in\{z\in\mathbb{C}~|~|\arg(z)|<\pi-2\alpha\}$. Then for any $\varepsilon>0$, there exists an algorithm that given an $n$-vertex graph $G$ of maximum degree at most $\Delta$, computes a multiplicative $\varepsilon$-approximation\footnote{A multiplicative \emph{$\varepsilon$-approximation} to a nonzero complex number $e^{a}$ is a number $e^{b}$ such that $|a-b|\leq \varepsilon$.} to $Z_{L(G)}(\xi,b)$  in time polynomial in $n/\varepsilon$.
\end{corollary}

\subsection*{Edge cover polynomial}
The second polynomial we consider is the \emph{edge cover polynomial}, which was introduced for graphs in \cite{akbari2013edge}. 
We consider here the obvious extension to hypergraphs.

Let $\mathcal{H}=(V,E)$ be a hypergraph. A subset of edges $F\subseteq E$ is called an \emph{edge cover} if each vertex of $\mathcal{H}$ is contained in at least one edge of $F$. 
We define the \emph{edge cover polynomial} of $\mathcal{H}$ as
\[
    \mathcal E(\mathcal{H},z)=\sum_{\textrm{$F\subseteq E$ edge cover}}z^{|F|}.
\]

In \cite{csikvari2011roots} it was proved that all the complex zeros of the edge cover polynomial of ordinary graphs are contained in the open disk of radius $\frac{(2+\sqrt{3})^2}{1+\sqrt{3}}$. Moreover, they showed that if the minimum degree is large enough, then the zeros are contained in the open disk of radius $4$. The authors conjectured that the actual bound will be $4$ for all graphs, cf.~\cite[Conjecture~5.1]{csikvari2011roots}.

Here we confirm this conjecture (see Section~\ref{sec:edge cover} for the proof):

\begin{theorem}\label{thm:edge cover}
Let $\mathcal{H}$ be a hypergraph with largest edge of size $k$ without isolated vertices. Then
\begin{itemize}
\item[(i)]  $\mathcal E(\mathcal{H},z)\neq 0$ if $|z|>2^{k}$,
\item[(ii)] if moreover $k=2$, then $\mathcal E(G,z)\neq 0$ if $|z|\geq 4$.
\end{itemize}
\end{theorem}

In fact we can prove a more refined result, which for graphs gives a zero-free region containing  the positive real line. We state here only the result for graphs, see Section~\ref{sec:edge cover} for the extension to hypergraphs and its proof.

\begin{theorem}\label{thm:edge cover strong}
Let $G$ be graph without isolated vertices. Then all roots of $\mathcal{E}(G,z)$ are contained in the set $\{-(1-\alpha)^2\mid |\alpha|\leq 1\}.$
\end{theorem}

See Figure~\ref{fig:cardioid} below for a picture of the set in Theorem~\ref{thm:edge cover strong}.
\begin{figure}[ht]
 \centering
 \includegraphics[width=0.4\textwidth]{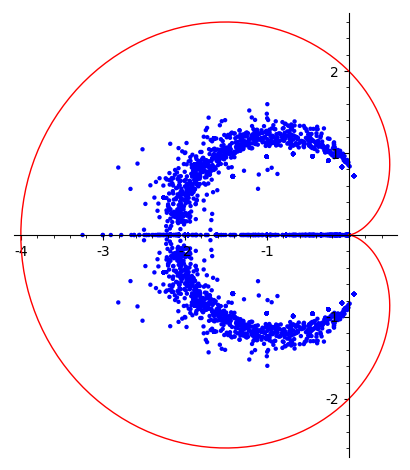}
 \caption{Roots of edge cover polynomial of some graphs on $10$ vertices and the boundary of the set $\{-(1-\alpha)^2\mid |\alpha|\leq 1\}.$}
\end{figure}\label{fig:cardioid}

Combining the theorem above with Barvinok's method~\cite{barvinok2016} and the improvement from~\cite{patel2017}, we obtain a fully polynomial time approximation scheme for approximating the edge cover polynomial on the the complement of the set $\{-(1-\alpha)^2\mid |\alpha|\leq 1\}$ on bounded degree graphs. To do this one needs to interpolate from `infinity'. Equivalently, one can interpolate the independence polynomial of the dual hypergraph (see below) from zero.
We omit further details, but refer to~\cite{barvinok2016,patel2017}, noting that in ~\cite{patel2017} hypergraphs are not considered but the extension of the approach from~\cite{patel2017} to hypergraphs is laid out in~\cite{liu2019ising}.
We note that on the real line this a gives a completely different algorithm than the one in~\cite{lin2014simple,liu2014fptas}, which is based on decay of correlations.
See also~\cite{bezakova2009sampling,huang2016canonical} for randomized algorithms based on Markov chains.

\subsubsection*{Independent sets}
It is important to note that the notion of an edge cover of a hypergraph $\mathcal{H}$ is strongly related to independent sets of hypergraphs. We call a set of vertices $A\subseteq V$ an \emph{independent set}, if no subset of $A$ forms an edge of $\mathcal{H}$ (in~\cite{berge} this is called a \emph{weak independent set}). In other words, a set $A$ is independent if and only if any edge of $\mathcal{H}$ contains at least one vertex from $V\setminus A$.  Such a set is called a  \emph{vertex cover}. 
Let us denote the \emph{independence polynomial} of $\mathcal{H}$ by
\[
    I(\mathcal{H},z)=\sum_{A\subseteq V \textrm{independent}} z^{|A|}.
\]
A natural way to describe a vertex cover of a hypergraph $\mathcal{H}$ is to consider an edge cover in the \emph{dual hypergraph} $\mathcal{H}^T$ with vertex set $E$ and edges $\{\{e\in E~|~ v\in e\} ~|~ v\in V\}$. Then it is not hard to see that
\[
    I(\mathcal{H},z)=z^{|V|}\mathcal{E}(\mathcal{H}^T,1/z).
\]
Thus we have the following corollary.

\begin{corollary}
 Let $\mathcal{H}$ be a hypergraph of degree at most $\Delta$. If $|z|<2^{-\Delta}$, then 
 \[
    I(\mathcal{H},z)\neq 0.
 \]
\end{corollary}

\subsubsection*{Domination polynomials}
As an additional application we obtain a  bound on the roots of the total domination and domination polynomial that is independent of the number of vertices.

For a graph $G$ a set $S\subseteq V(G)$ is called a dominating set if for every $u\in V(G)$ we have either $u\in S$ or there exists a neighbor $v\in S$ of $u$. Let $d_k(G)$ denote the number of dominating sets of size $k$ in $G$. The domination polynomial is defined as $D(G,z)=\sum_kd_k(G)z^k$.  See~\cite{alikhani2009introduction,akbari2010zeros,kotek2012recurrence} for further details.

Similarly, for a graph $G$ a set $S\subseteq V(G)$ is called a total dominating set if for every $u\in V(G)$ there exists a neighbor $v\in S$ of $u$. Let $d^{(t)}_{k}(G)$ denote the number of dominating sets of size $k$ in $G$. The total domination polynomial is defined as $D_t(G,z)=\sum_kd^{(t)}_{k}(G)z^k$. See~\cite{chaluvaraju2014total,dod2016graph} for further details.

In \cite{oboudi2016roots} it is shown that $D(G,z)$ has all its complex zeros in a disk of radius $\sqrt[\delta+1]{2^n-1}$  around $-1$. (Here $\delta$ denotes the minimum degree.)
 In \cite{jafari2020roots} the authors showed a similar bound for $D_t(G,z)$, namely all the complex zeros are in the disk of radius $\sqrt[\delta]{2^n-1}$ around $-1$.
Observe that both bounds depend on the number of vertices of the graph.

For a graph $G$ and a  vertex $v$ let $N_G(v)$ denote the set of neighbors of $v$, that is, $N_G(v)=\{u\in V ~|~ (u,v)\in E\}.$ Let $N_G[v]=N_G(v)\cup \{v\}$, this is called the closed neighborhood of $v$.  Let us define the hypergraph $\mathcal{D}_G$ (resp. $\mathcal{D}_{G,t}$)   on the vertex set $V(G)$  with edges $\{N_G[v]~|~v\in V(G)\}$ (resp. $\{N_G(v)~|~v\in V(G)\}$). Now observe that the (total) domination polynomial of $G$ is an edge cover polynomial of 
$\mathcal{D}_G$ (resp. $\mathcal{D}_{G,t}$), that is, $D(G,z)=\mathcal{E}(\mathcal{D}_G,z)$ and $D_{t}(G,z)=\mathcal{E}(\mathcal{D}_{G,t},z)$.
By specializing Theorem~\ref{thm:edge cover} we obtain that the roots of $D(G,z)$ (resp. $D_{t}(G,z)$) are contained in a disk of radius $2^{\Delta(G)+1}$ (resp. $2^{\Delta(G)}$) around $0$, assuming the graph $G$ has no isolated vertices. (Here $\Delta$ denotes the maximum degree.)

\begin{corollary}
Let $G$ be a graph without isolated vertex and with maximum degree $\Delta$. Then the roots of the domination polynomial of $G$ are contained in the disk of radius $2^{\Delta(G)+1}$.
The roots of the total domination polynomial of $G$ are contained in the disk of radius $2^{\Delta(G)}$.
 \end{corollary}

\subsection*{Organization}
In the next section we recall some notation and the main result of Wagner~ \cite{wagner2009}.
In Section~\ref{sec:edge cover} we prove our results for the edge cover polynomial, and in Section~\ref{sec:Ising} we prove our results for the Ising model. In the final section we close with an open question and some remarks.


\section{Preliminaries}
Let us recall some notations and results of Wagner \cite{wagner2009}.

Let $\mathcal{A}\subset \mathbb{C}$.
We say that a multivariate polynomial $p$ with variables $z_1,\dots,z_n$ is $\mathcal{A}$-nonvanishing, if  either $p$ is constant zero or $p(z_1,\dots,z_n)\neq 0$ if all $z_i\in \mathcal{A}$.

Also we denote
\begin{itemize}
 \item the sector $\mathcal{S}[\theta]=\{z\in\mathbb{C}~|~|\arg(z)|<\theta\}$ for some $0\le\theta< \pi$;
 \item the open interior of disk $\kappa\mathcal{D}=\{z\in\mathbb{C}~|~|z|<\kappa\}$ for some $0<\kappa$;
 \item the open exterior of a disk $\rho\mathcal{E}=\{z\in\mathbb{C}~|~|z|>\rho\}$ for some $0<\rho<\infty$.
\end{itemize}

In what follows we will use the degree sequence and the degree of a vertex. To emphasize the difference, we will use $\deg(\mathcal{H})$ for the degree sequence of a hypergraph $\mathcal H$, and $d_{\mathcal H}(v)$ for the degree of a vertex $v$. 
Let us fix a hypergraph $\mathcal{H}=(V,E)$. Then associate to each vertex $v$ of $\mathcal{H}$ a sequence of complex numbers $u^{(v)}=(u^{(v)}_0,\dots,u^{(v)}_{d_\mathcal{H}(v)})$ and to every edge $e$ associate a complex number $\lambda_e$. We define the (multivariate) subgraph counting polynomial of $\mathcal H$ with variables $x_v$, $v\in V$, as
\[
  Z_\mathcal{W}(\mathcal{H},\lambda,u;x)=\sum_{F\subseteq E}\lambda^Fu_{\textrm{deg}(F)}x^{\textrm{deg}(F)},
\]
where
$\lambda^F=\prod_{e\in F}\lambda_e$, $u_{\textrm{deg}(F)}=\prod_{v\in V}u^{(v)}_{d_F(v)}$ and $x^{\textrm{deg}(F)}=\prod_{v\in V}x_v^{d_F(x)}$.

The strategy in \cite{wagner2009} to obtain a zero-free region for $Z_\mathcal{W}(\mathcal{H},\lambda,u;x)$ is based on properties of two other polynomials. The first is the \emph{base polynomial} of $\mathcal{H}$, defined as,
\[
    \Omega(\mathcal{H},\lambda,x)=\prod_{e\in E}\left(1+\lambda_e\prod_{v\in e}x_v\right)
\]
and the other is the \emph{key polynomial} of a vertex $v$, defined as
\[
    K^{(v)}(z)=\sum_{i=0}^{d_\mathcal{H}(v)}\binom{d_\mathcal{H}(v)}{i} u_i^{(v)}z^i.
\]

Now we are ready to state Wagner's theorem.

\begin{theorem}[Theorem~3.2 of \cite{wagner2009}]\label{th:wagner}
Let $\mathcal{H}$, $u$, and $\lambda$ be defined as above. 
\begin{itemize}
\item[(i)] If $\Omega(\mathcal{H},\lambda;x)$ is $\mathcal{S}[\pi/2]$-nonvanishing and for each vertex $v$, $K^{(v)}(z)$ is $\mathcal{S}[\pi-\alpha]$-nonvanishing, then $Z_\mathcal{W}(\mathcal{H},\lambda,u;x)$ is $\mathcal{S}[\pi/2-\alpha]$-nonvanishing.
\item[(ii)] If $\Omega(\mathcal{H},\lambda;x)$ is $\kappa\mathcal{D}$-nonvanishing  and for each vertex $v$, $K^{(v)}(z)$ is $\rho\mathcal{D}$-nonvanishing, then $Z_\mathcal{W}(\mathcal{H},\lambda,u;x)$ is $\kappa\rho\mathcal{D}$-nonvanishing. 
\item[(iii)] If $\Omega(\mathcal{H},\lambda;x)$ is  $\kappa\mathcal{E}$-nonvanishing and for each vertex $v$, $K^{(v)}(z)$ is $\rho\mathcal{E}$-nonvanishing and of degree $d_{\mathcal{H}}(v)$, then $Z_\mathcal{W}(\mathcal{H},\lambda,u;x)$ is  $\kappa\rho\mathcal{E}$-nonvanishing. 
\end{itemize}
\end{theorem}

We note that this result is only stated for graphs in~\cite{wagner2009}, but the extension to hypergraphs that we present here is straightforward. For convenience of the reader we will provide a proof, closely following Wagner's proof for the graph case.
Before we start, we would further like to make two remarks.

\begin{remark}\label{rem:no angle for hypergraps}
From the definition of $\Omega(\mathcal{H},\lambda,x)$ it follows that if each hyperedge has size at least $2$, then $\Omega(\mathcal{H},\lambda,x)$ is $\mathcal{S}[\pi/2]$-nonvanishing if and only if each $\lambda_e\ge 0$ and the size of each hyperedge equals $2$. 
Therefore part (i) of the theorem only applies to ordinary graphs.
\end{remark}

\begin{remark}\label{rem:multivariate to univariate}
 Another useful observation is the case when $\mathcal{H}$ is a $k$-uniform hypergraph. Using the substitution $x_v=z^{1/k}$ for each vertex $v$ of $\mathcal{H}$, the polynomial
\[
    Z_{\mathcal{W}}(\mathcal{H},\lambda,u;x)=\sum_{F\subseteq E}\lambda^Fu_{\textrm{deg}(F)}z^{|F|}
\]
is a one variable polynomial, since in a $k$-uniform hypergraph $\sum_{v\in V}d_F(v)=k|F|$.
\end{remark}

To prove Theorem~\ref{th:wagner} we will need the following lemma.

\begin{lemma}[Schur-Szeg\H o, Proposition 2.4(b) and (c) of~\cite{wagner2009}] \label{lem:wagner}
Let $P(z)=\sum_jc_jz^j$ and $K(z)=\sum_{j=0}^d\binom{d}{j}u_jz^j$ be polynomials in one complex variable, with $\deg P\leq d$. The Schur-Szeg\H o composition of polynomials $P(z)$ and $K(z)$ is the polynomial $Q(z)=\sum_{j=0}^du_jc_jz^j$. For any $\kappa>0$ and $\rho>0$, if $P(z)$ is $\rho\mathcal{D}$-nonvanishing and $K(z)$ is $\kappa\mathcal{D}$-nonvanishing, then $Q(z)$ is $\kappa\rho\mathcal{D}$-nonvanishing. Similarly, if $P(z)$ is $\rho\mathcal{E}$-nonvanishing and $K(z)$ is $\kappa\mathcal{E}$-nonvanishing and $\deg(K)=d$, then $Q(z)$ is $\kappa\rho\mathcal{E}$-nonvanishing.
\end{lemma}

\begin{proof}[Proof of Theorem~\ref{th:wagner}]
By Remark~\ref{rem:no angle for hypergraps} above, part (i) is covered by Theorem~3.2 of ~\cite{wagner2009}. 
We therefore focus on the case that $\Omega(\mathcal{H},\lambda;x)$ is $\kappa\mathcal{D}$-nonvanishing and each key polynomial $K^{(v)}(z)$ is $\rho\mathcal{D}$-nonvanishing. The proof for the case where $\mathcal{D}$ is replaced by $\mathcal{E}$ follows along exactly the same lines.

We identify the vertex set $V$ with $\{1,\ldots,n\}.$ We define a sequence of polynomials, $F_0(x),F_1(x),\ldots,F_n(x)$ as follows.
We set $F_0(x):=\Omega(\mathcal{H},\lambda,x)$. For all $1\leq i\leq n$ we let $F_i(x)$ to be obtained as the Schur-Szeg\H o composition of $F_{i-1}(x)$ and the $i$th key polynomial in the variable $x_i$, $K_i(x_i)$ (the remaining variables being absorbed in the coefficients).
By induction one has
\[
F_i(x)=\sum_{F\subseteq E} \lambda^F \left(\prod_{j=1}^i u^{(j)}_{d_F(j)}\right) x^{\deg(F)},
\]
implying that $F_n(x)=Z_\mathcal{W}((\mathcal{H},\lambda,u;x))$.

We next show by induction that if $\xi_1,\ldots,\xi_n$ are such that 
\[\xi_j \in \kappa\rho \mathcal{D} \text{ if } j<i \text{ and } \xi_j\in \rho\mathcal{D} \text{ if } j>i,
\]
then $F_{i-1}(\xi_1,\ldots,\xi_{j-1},x_i,\xi_{i+1},\ldots,\xi_n)$ is $\rho\mathcal{D}$-nonvanishing.
The base case $i=1$ follows from the assumption. 
By Lemma~\ref{lem:wagner}  we immediately obtain the induction step.
To see that $F_n(x)$ is $\kappa\rho\mathcal{D}$-nonvanishing, we apply Lemma~\ref{lem:wagner} once more to $F_{n-1}(\xi_1,\ldots,\xi_{n-1},x_n)$ and $K_n(x_n)$, for any choice of $\xi_1,\ldots,\xi_{n-1}\in \kappa\rho\mathcal{D}$, which is $\kappa\mathcal{D}$-nonvanishing by the above, to obtain the desired result.
\end{proof}

Our main goal will be to express the partition function of the Ising-model of a line graph and the edge cover polynomial as a subgraph counting polynomial. 
We start with the edge cover polynomial, as this is easiest one.
%

\section{Edge cover polynomial}\label{sec:edge cover}
We start by giving two proofs of Theorem~\ref{thm:edge cover}(i), by expressing the edge cover polynomial as a subgraph counting polynomial in two different ways.
After this we prove Theorem~\ref{thm:edge cover strong} and we conclude with proving that for graphs $-4$ cannot be a root of the edge cover polynomial thereby concluding the proof of Theorem~\ref{thm:edge cover}(ii).

\subsection{First proof of Theorem~\ref{thm:edge cover}(i)}
\begin{lemma}\label{lemma:ecwagner}
 For a hypergraph $\mathcal{H}$, the edge cover polynomial at $\xi$ can be expressed as
 \[
    \mathcal E(\mathcal{H},\xi)=Z_\mathcal{W}(\mathcal{H},\xi,(0,1,\dots,1); 1).
 \]
\end{lemma}

\begin{proof}
 We simply have to  check the definition of the subgraph counting polynomial, that is,
 \[
    Z_\mathcal{W}(\mathcal{H},\xi,(0,1,\dots,1); 1)=\sum_{F\subseteq E}\xi^{|F|}\prod_{v\in V(\mathcal{H})}\mathbf{1}_{d_F(v)>0}=\mathcal E(\mathcal{H},\xi).
 \]
\end{proof}

To apply Wagner's theorem, we have to investigate the location of zeros of the key polynomials. Let $L_d(z)=(1+z)^d$ and $K_d(z)=(1+z)^d-1$.

\begin{lemma}\label{lemma:eckey}
 For any $d\ge 0$, the polynomials 
 \[L_d(z)=(1+z)^d \ \ \ \ \mbox{and} \ \ \ K_d(z)=\sum_{i=1}^d\binom{d}{i}z^i=(1+z)^d-1\]
 are $2\mathcal{E}$-nonvanishing.
\end{lemma}

\begin{proof}
 The statement is trivial for $L_{d}(z)$. On the other hand, the roots of $K_d(z)$ are translations of the $d$-th root of unity by $1$. 
 Since the $d$-th roots of unity form vertices of a regular $d$-gon, therefore by simple geometric argument, we obtain the desired statement.
\end{proof}

Let us fix a value $\xi\in\mathbb{C}$, such that $\kappa=|\xi|>2^k>1$. 

\begin{lemma}
 The base polynomial, 
 \[
  \prod_{e\in E}\left(1+\xi\prod_{v\in e}x_v\right),
 \]
 is $\kappa^{-1/k}\mathcal{E}$-nonvanishing.
\end{lemma}

\begin{proof}
 This is clear since  if each $x_v$ has absolute value at least $\kappa^{-1/k}$, then
 \[
    \left|\xi\prod_{v\in e}x_v\right|>\kappa \prod_{v\in e} \kappa^{-1/k}=\kappa^{1-|e|/k}\ge 1.
 \]
\end{proof}

Now we are ready to prove Theorem~\ref{thm:edge cover}(i). 
We would like to show that $\mathcal{E}(\mathcal{H},\xi)\neq 0$.
Consider the polynomial $Z_\mathcal{W}(\mathcal{H},\xi,(0,1,\dots,1),z)$. 
We will use Theorem~\ref{th:wagner} to show that the subgraph counting polynomial is not zero at $z=1$. 
Indeed, as the key polynomials are $2\mathcal{E}$-nonvanishing, we get from Theorem~\ref{th:wagner} that $Z_\mathcal{W}(\mathcal{H},\xi,(0,1,\dots,1),z)$ is $2(\kappa^{-1/k})\mathcal{E}$-nonvanishing, that is $(1-\varepsilon)\mathcal{E}$-nonvanishing for some $\varepsilon>0$. 
 As we are interested in the value of this polynomial at $z=1$, and since $1\in (1-\varepsilon)\mathcal{E}$, we therefore obtain,
 \[
    0\neq Z_\mathcal{W}(\mathcal{H},\xi,(0,1,\dots,1),1)=\mathcal{E}(\mathcal{H},\xi),
 \]
 as desired.

\subsection{Second proof of Theorem~\ref{thm:edge cover}(i)} 
It will be convenient for us to define for a hypergraph $\mathcal{H}=(V,E)$ of largest edge size $k$, its uniformization, $\widehat{\mathcal{H}}$, by adding new extra vertices to edges, in a way that we obtain a $k$-uniform hypergraph with same number of edges. Let the set of new vertices be denoted by $S$ and the set of edges of $\widehat{\mathcal{H}}$ by $\widehat{E}$.

\begin{lemma}\label{lemma:ecwagner2}
 For a hypergraph $\mathcal{H}$ the edge cover polynomial can be expressed as
 \[
    \mathcal E(\mathcal{H},z)=Z_\mathcal{W}(\widehat{\mathcal{H}},1,u;z^{1/k}),
 \]
 where $u^{(v)}=(0,1,1,\dots,1)$ if $v\notin S$ and $u^{(v)}=(1,\dots,1)$ otherwise.
\end{lemma}
\begin{proof}
 Observe that a subset of edges $E$ in $\mathcal{H}$ is an edge cover if and only if the corresponding edges in $\widehat{\mathcal{H}}$ covers $V(\mathcal{H})\setminus S$ as well. Thus the edge covering polynomial of $\mathcal{H}$ and the ``relaxed edge cover polynomial'' of $\widehat{\mathcal{H}}$ are the same.
The lemma is now an immediate corollary of the definition of the edge cover polynomial and Remark~\ref{rem:multivariate to univariate} in the previous section. 
\end{proof}

In order to apply Wagner's theorem, we have to investigate the location of zeros of the key polynomials  $L_{d}(z)=(1+z)^d$ and $K_d(z)=\sum_{i=1}^d\binom{d}{i}z^i=(1+z)^d-1$.
We already proved the relevant properties of these polynomials in Lemma~\ref{lemma:eckey}. 
Using this, we obtain by Theorem~\ref{th:wagner}, that
 \[
    Z_{\mathcal{W}}(\widehat{\mathcal{H}},1,u;z)=\sum_{F\subseteq \widehat{E} \textrm{ covers $V(\mathcal{H})\setminus S$}} z^{\sum_{v\in\widehat V}d_F(v)}=\sum_{F\subseteq \widehat{E} \textrm{ covers $V(\mathcal{H})\setminus S$}} z^{k|F|}
 \]
 is $2\mathcal{E}$-nonvanishing. By substituting $z\mapsto z^{1/k}$, we obtain a polynomial that is $2^{k}\mathcal{E}$-nonvanishing, as desired.

\subsection{The Cardioid-like region}
In this section, we will strengthen Theorem~\ref{thm:edge cover}. 
We will not use Theorem~\ref{th:wagner}, but a similar technique to find a zero-free region for the subgraph-counting function. 

The following lemma will play the role of Asano-contraction in the main proof. This lemma is  a slight modification of \cite[Lemma~7]{guo2020zeros}.

\begin{lemma}\label{lemma:asano}
 Let $p(z)=\sum_{k=0}^da_kz^k$ such that $a_d\neq 0$.  Assume that $p(z)\neq 0$ if $z\notin K$ for some $K\subseteq \mathbb{C}$ closed set. Then
 \[
    q(z)=a_dz+a_0
 \]
 has its (unique) root contained in $(-1)^{d+1}K^d$.
\end{lemma}
\begin{proof}
 Let $p(z)=a_d(z+\xi_1)\dots(z+\xi_d)$, where $\xi_i\in -K$.
 If we denote the root of $q(z)$ by $z_0$, then
 \begin{align*}
    z_0=-\frac{a_0}{a_d}=-\frac{a_d\xi_1\dots\xi_d}{a_d}\in(-1)^{d+1}K^d.
 \end{align*}
\end{proof}
We will call the previous polynomial transformation the \emph{Asano-contraction of $p(z)$ over the variable $z$ of degree $d$}.

The idea of the proof is that we use Asano-contraction iteratively for a rightly chosen multivariate polynomial. The issue that could occur is that the resulting polynomial does not have the correct degree so that we cannot apply Asano-contraction again. To rule out this case, we have to relax the definition of the edge cover polynomial by saying there are vertices where we can use any number of edges similarly to the subgraph counting polynomial of the previous subsection.

\begin{definition} Let $\mathcal{H}$ be a hypergraph with $E(\mathcal{H})=\{e_1,\dots,e_m\}$ edges, and let $S$ be a subset of $V(\mathcal{H})$. Then we define
 \begin{itemize}
  \item the \emph{base-polynomial}
  \[
    \Omega_{\mathcal H,S}(z_{e_1},\dots,z_{e_m})=\prod_{v\in S}\left(\prod_{e:v\in e}(1+z_e)\right)\prod_{v\notin S}\left(\prod_{e:v\in e}(1+z_e)-1\right);
  \]
  \item for $0\le k\le m$ the \emph{intermediate polynomials} as
  \[
    P_{\mathcal H,S,0}(\mathbf{z})=\Omega_{\mathcal{H},S}(\mathbf{z}),
  \]
  and $P_{\mathcal{H},S,k}(\mathbf{z})$ as the Asano-contraction of $P_{\mathcal{H},S,k-1}(\mathbf{z})$ over the variable $z_{e_k}$ of degree $|e_k|$;
  \item the \emph{edge-cover polynomial of $\mathcal{H}$ relaxed over $S$} as
  \[
    \mathcal E(\mathcal{H},S,z)=P_{\mathcal{H},S,m}(z,\dots,z);
  \]
  \item $B=\{\alpha\in\mathbb{C}~|~|\alpha+1|> 1\}$, and $T_k=\mathbb{C}\setminus \{(-1)^{k+1}\prod_{i=1}^k\alpha_i~|~\alpha_1,\dots,\alpha_k\notin B\}$, in particular, $T_1=B$;
  \item Notation: 
  $T_{\mathcal{H},k}=T_{|e_1|}\times\dots\times T_{|e_k|}\subseteq \mathbb{C}^{k}$.
 \end{itemize}
\end{definition}

\begin{remark}
 It is important to understand the meaning of $P_{\mathcal{H},S,m}(z_{e_1},\dots,z_{e_m})$. Note that it is a multilinear polynomial since after an application of the Asano contraction to the variable $z_e$ the resulting polynomial will be linear in this variable. So this polynomial can be written as 
 $$\sum_{F\subseteq E(\mathcal{H})}a_F\prod_{e\in F}z_e.$$
 Now let us understand the coefficient $a_F$. This means that whenever $e\in F$ we chose the degree $|e|$ term from the previous multivariate polynomial, and whenever $e\notin F$ we chose the constant term. Now observe that if $v\notin S$, then there is no constant term in $\prod_{v\in e}(1+z_e)-1$. This means that we should choose  at least one edge covering $v$ into $F$, otherwise $a_F=0$. For $v\in S$ there is no such requirement since there is a constant term in $\prod_{v\in e}(1+z_e)$. So $a_F=1$ if the elements of $F$ cover every vertex not in $S$. They may cover some vertices from $S$, but they do not need to. So the meaning of 
 \[
    \mathcal E(\mathcal{H},S,z)=P_{\mathcal{H},S,m}(z,\dots,z)
  \]
  is that it counts the edge sets $F$ with weight $z^{|F|}$ if it covers all vertices not in $S$. Thus $\mathcal{E}(\mathcal{H},\emptyset,z)$ is the classical edge cover polynomial of the hypergraph $\mathcal{H}$.
\end{remark}

\begin{theorem}
Let $\mathcal{H}=(V,E)$ be a hypergraph. For any set $S\subseteq V$ containing all isolated vertices of $\mathcal H$, and $k\ge 0$, the polynomial 
 \[
    P_{\mathcal{H},S,k}(\mathbf{z})
 \]
 is non-zero and $T_{\mathcal{H},k}\times B^{m-k}$-nonvanishing.
\end{theorem}
\begin{proof}
We will prove the statement by induction on $k$.
If $k=0$, then the statement is trivial for any $S\subseteq V(G)$, since for any vertex the polynomials
\[
    \prod_{v\in e}(1+z_e) \textrm{ and } \prod_{v\in e}(1+z_e)-1
\]
are $B^m$-stable.

Let us assume that $k\ge 1$. Then, by induction, we have $P_{\mathcal{H},S,k-1}(\mathbf{z})$ is $T_{\mathcal{H},k-1}\times B^{m-k+1}$-nonvanishing non-zero polynomial. Fix $(\xi_1,\dots,\xi_{k-1})\in T_{\mathcal{H},k-1}$, and $\tau_{k+1},\dots,\tau_{m}\in B$. Then,
\[
    p(z)=P_{\mathcal{H},S,k-1}(\xi_1,\dots,\xi_{k-1},z,\tau_{k+1},\dots,\tau_{m})
\]
is a non-zero $B$-stable polynomial or the zero-polynomial. 

We claim that $p(z)$ is a  polynomial of degree $|e_k|$. This is true, since the coefficient of $z^{|e_k|}$ in $p(z)$ is exactly
\[
    P_{\mathcal{H}-e_k,S\cup e_k,k-1}(\xi_1,\dots,\xi_{k-1},\tau_{k+1},\dots,\tau_m),
\]
which is not the zero polynomial by induction. (Since new additional isolated vertices of $\mathcal{H}-e_k$ are in $e_k$.)

If we denote 
\[
    q(z)=P_{\mathcal{H},S,k}(\xi_1,\dots,\xi_{k-1},z,\tau_{k+1},\dots,\tau_{m}),
\]
then $q(z)$ is exactly the Asano-contraction of $p(z)$ over $z$ of degree $|e_k|$. By Lemma~\ref{lemma:asano} we have that $q(z)\neq 0$ if $z\in T_{|e_k|}$.

Thus, we proved that for any choice of $(\xi_1,\dots,\xi_{k-1},\xi_k)\in T_{\mathcal{H},k}$ and $\tau_{k+1},\dots,\tau_{m}\in B$, the polynomial
\[
    P_{\mathcal{H},S,k}(\xi_1,\dots,x_k,\tau_{k+1},\dots,\tau_{m})\neq 0,
\]
as desired.
\end{proof}

Theorem~\ref{thm:edge cover strong} follows directly from this result, taking $S=\emptyset$ and realizing that the set $T_2$ is equal to the complement of the set $\{-(1-\alpha)^2\mid |\alpha|\leq 1\}$, see the next lemma.

\begin{lemma}\label{lem:cardioid}
 The set $T_2\subseteq \mathbb{C}$ is exactly
$
    \mathbb{C}\setminus \{-(1-\alpha)^2 ~|~ |\alpha|\le 1\}
$.
\end{lemma}

In the forthcoming proof $\measuredangle BAC$ means the angle at $A$ determined by the lines $BA$ and $AC$.

 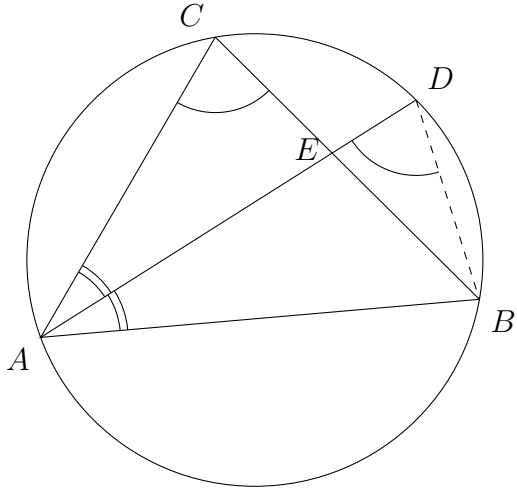
\begin{figure}[ht]
     \centering
    \begin{tikzpicture}[scale=0.6]
        \draw (0,0) circle[radius=5];
        \node[below left] at (200:5) {$A$};
        \coordinate (A) at (200:5);
        \node[below right] at (350:5) {$B$};
        \coordinate (B) at (350:5);
        \node[above left] at (100:5) {$C$};
        \coordinate (C) at (100:5);
        \node[above right] at (45:5) {$D$};
        \coordinate (D) at (45:5);
        
        \draw (C)--(A) pic[draw=black, angle eccentricity=1.2, angle radius=1cm]
            {angle=D--A--C};
        \draw pic[draw=black,angle eccentricity=1.2, angle radius=1.1cm]
            {angle=D--A--C};
        \draw (D) -- (A) -- (B);
        \draw pic[draw=black,angle eccentricity=1.2, angle radius=1.05cm]
            {angle=B--A--D};
        \draw pic[draw=black,angle eccentricity=1.2, angle radius=1.15cm]
            {angle=B--A--D};
        \draw (C)--(B);
        \draw[dashed] (D) --(B);
        
        \draw pic[draw=black,angle eccentricity=1.2, angle radius=1cm]
            {angle=A--C--B};
        \draw pic[draw=black,angle eccentricity=1.2, angle radius=1cm]
            {angle=A--D--B};
            
        \node (E) at (65:2.7) {$E$};
    \end{tikzpicture}

     \caption{Proof of the geometric fact from Lemma~\ref{lem:cardioid}.}
     \label{fig:fact}
 \end{figure}

\begin{proof}
 The statement is equivalent with the the following statement. Let $R=\{z ~|~ |(-1)-z|\leq 1\}$, then the sets $R_1=\{z^2 ~|~ z\in R\}$ and  $R_2=\{z_1z_2 ~|~ z_1,z_2\in R\}$ are the same. Clearly, $R_1\subseteq R_2$ so we only need to prove the opposite containment. Note that both $R_1$ and $R_2$ are star convex from the point $0$, so it is enough to prove that if $w\in R_2$, then there exists a $w'$ such that $\mathrm{arg} (w')=\mathrm{arg} (w)$ and $|w'|\geq |w|$.
 So we assume that $z_1,z_2\in R$ are on the boundary of $R$, and $w=z_1z_2$. We show that if we choose $z'$ to be the intersection of the boundary of $R$ with the angle bisector of the angle determined by $z_1,0,z_2$, then $w'=z'^2$ satisfies the above conditions.
 
 To see this we need the following geometric fact. Let $ABC$ be a triangle and let $D$ be the intersection of the circumscribed circle of $ABC$ and the angle bisector of $\angle BAC$. 
 Let $E$ be the intersection of the angle bisector of $\angle BAC$ and the side $BC$. 
 Then $|AB|\cdot |AC|=|AD|\cdot |AE|$. 
 This is because the triangles $ABD$ and $AEC$ are similar: $\measuredangle BAD=\measuredangle EAC$ since $AD$ is an angle bisector, and $\measuredangle ADB=\measuredangle ACE$ by the inscribed angle theorem (see Figure~\ref{fig:fact} for a picture describing this).
 Hence
 $|AB|\cdot |AC|=|AD|\cdot |AE|\leq |AD|^2$. 
 
 Applying this to $A=0$, $B=z_1$, $C=z_2$, $D=z'$ we get the claim that for $w=z_1z_2$ and  $w'=z'^2$ we have $\mathrm{arg}(w')=\mathrm{arg}(w)$ and $|w'|\geq |w|$.

\end{proof}

 

\subsection{The evaluation at $z_0=-4$}

The main goal of this subsection is to finish the proof of Theorem~\ref{thm:edge cover}(ii) and thereby confirming Conjecture~5.1 of \cite{csikvari2011roots}. 
By Theorem~\ref{thm:edge cover strong} a root of the edge cover polynomial can have absolute value $4$ only if it is equal to $-4$. As we will show, this case cannot occur.

We need the following lemma.

\begin{lemma}
 Let $G=(V,E)$ be a multigraph on $m$ edges and let $S\subset V$ be a set of vertices containing all isolated vertices of $G$. 
 If $\mathcal E(G,S,-4)=0$ and $e=(u,v)\in E$, then either $\mathcal E(G-e,S\cup\{u,v\},-4)=0$ or $e$ is not a loop and $\mathcal E(G/e,S,-4)=0$.
\end{lemma}

\begin{proof}
 Suppose for contradiction that $\mathcal E(G-e,S\cup\{u,v\},-4)\neq 0$ and if $e$ is not a loop, then $\mathcal E(G/e,S,-4)\neq 0$. Assume that the edges are ordered in a way such that $e_m=e$. 
 To proceed, we have to discuss a few cases depending the size of $e$, and on the number of endpoints of $e=(u,v)$ being in $S$.
 
 If $u=v$, that is, $e$ is a loop, then
 the polynomial 
 \[
    P_{G,S,m}(-4,\dots,-4,z)
 \]
 is a linear $T_1=B$-nonvanishing polynomial that is zero at $-4$. This could only happen, if this is the zero polynomial, in which case the `main coefficient' $\mathcal{E}(G-e,S\cup\{u\},-4)=0$, a contradiction.
 
 For the remainder we may assume that $u\neq v$.
 If $u,v\in S$, then we have
 \[
    0=\mathcal E(G,S,-4)=(-4+1)\cdot \mathcal E(G-e,S,-4).
 \]
 Thus $\mathcal E(G-e,S\cup\{u,v\},-4)=0$, a contradiction. 
 
 So we may assume that not both $u$ and $v$ are contained in $S$. 
  The polynomial 
 \[
    q(z)=P_{G,S,m}(-4,\dots,-4,z),
 \]
 is a non-zero linear polynomial with a zero at $-4$. 
 Indeed, since the main coefficient is $\mathcal{E}(G-e,S\cup\{u,v\},-4)\neq 0$. 
 Next, consider the polynomial
 \[
    p(z)=P_{G,S,m-1}(-4,\dots,-4,z).
 \]
This is a polynomial of degree $2$, which is $B$-nonvanishing. 
Since $q(z)$ is the Asano-contraction of $p(z)$, the two zeros $\xi_1,\xi_2$ of $p(z)$ satisfy $\xi_1\xi_2=4$. On the other hand, we know that $p(z)$ is a $B$-nonvanishing polynomial, therefore, $\xi_1=\xi_2=-2$.
In particular, the constant term and the linear term of $p(z)$ are equal. 
We obtain that the constant term is equal to $\mathcal E(G-e,S,-4)$ and the linear term is equal to $\mathcal E(G-e,S\cup\{u\},-4)+\mathcal E(G-e,S\cup\{v\},-4)$.  
Therefore we obtain in case $u,v\notin S$,  that $\mathcal E(G-e,S,-4)$ equals
\begin{align*}
&\mathcal E(G-e,S\cup\{u\},-4)+  \mathcal E(G-e,S\cup\{v\},-4)
    \\
 =&\left( \mathcal E(G-e,S,-4)+ \mathcal E(G-u,S,-4)\right)+\left( \mathcal E(G-e,S,-4)+ \mathcal E(G-v,S,-4)\right).
 \end{align*}
 From which it follows that
 \[0= \mathcal E(G-e,S,-4)+ \mathcal E(G-u,S,-4)+  \mathcal E(G-v,S,-4)= \mathcal E(G/e,S,-4),
 \]
 as desired.
Otherwise, we may assume by symmetry, that $u\in S$ and $v\notin S$, in which case we obtain that $\mathcal E(G-e,S,-4)$ equals
 \begin{align*}
&\mathcal E(G-e,S\cup\{u\},-4)+\mathcal E(G-e,S\cup\{v\},-4)
\\
=&\mathcal E(G-e,S,-4)+\mathcal E(G-e,S\cup\{u,v\},-4),
 \end{align*}
 implying that $\mathcal E(G-e,S\cup\{u,v\},-4)=0$, a contradiction.
 This finishes the proof.
\end{proof}

As an immediate corollary we obtain that $-4$ cannot be a root of the edge cover polynomial.
\begin{corollary}
 Let $G=(V,E)$ be a multigraph and let $S\subseteq V$ be a set containing all isolated vertices of $G$. Then
 \[
  \mathcal  E(G,S,-4)\neq 0.
 \]
\end{corollary}
\begin{proof}
 For the sake of contradiction, assume that there exists such an example. 
 Let $G$ be a counterexample with the minimum number of edges. If $G$ has no edge, then $S=V$, thus, $\mathcal E(G,S,z)=1$. Therefore, we may assume that $G$ has an edge.
  Let $e\in E(G)$. Then, by the previous lemma, either $E(G-e,S\cup\{u,v\},-4)=0$ or $E(G/e,S,-4)=0$. But in each case, the number of edges is strictly less than the number of edges of $G$ that leads us to a contradiction.
\end{proof}

\section{The antiferromagnetic Ising-model on line graphs}\label{sec:Ising}
We begin by expressing the partition function of the Ising model on a line graph in terms of the underlying graph. 
\begin{lemma}\label{cl:express}
For a line graph $L=L(G)$, the partition function of the Ising model has the following form
\[
  Z_{L(G)}(z^2)=b^{|E(L)|}\sum_{F\subseteq E}\prod_{v\in V}b^{-\binom{d_G(v)-d_F(v)}{2}-\binom{d_F(v)}{ 2}}z^{d_F(v)}.
\]
\end{lemma}
\begin{proof}
 Since the vertex set of $L(G)$  coincides with the edges of $G$, therefore
 \[
  Z_{L(G)}(z^2,  b ) =
    \sum_{U \subseteq V(L(G))} z^{2\lvert U \rvert} \cdot b ^{\lvert \delta_{L(G)}(U) \rvert}=
    \sum_{F \subseteq E(G)}z^{2\lvert F \rvert} \cdot b ^{\lvert \delta_{L(G)}(F) \rvert}.
 \]
 
 To express $\delta_{L(G)}(F)$, observe that $L(G)$ can be covered by edge disjoint cliques $Q_v$, where each clique corresponds those edges of $G$ containing the vertex $v$. Therefore, to count the number of edges in $\delta_{L(G)}(F)$ we can take the disjoint union
 \[
    \delta_{L(G)}(F)=\bigcup_{v\in V(G)} E(Q_v)\cap \delta_{L(G)}(F).
 \]
 Let us assume that $Q_v$ is a clique of size $d$ and $|F\cap Q_v |= f$, i.e. $\deg_{G}(v)=d$ and $\deg_{F}(v)=f$. Then clearly,
 \[
    |E(Q_v)\cap \delta_{L(G)}(F)|=f(d-f)=\binom{d}{2}-\binom{f}{ 2}-\binom{d-f}{2}.
 \]
 Thus,
 \[
    |\delta_{L(G)}(F)|=\sum_{v\in V(G)}\left[ \binom{d_G(v)}{2}-\binom{d_F(v)}{2}-\binom{d_{G}(v)-d_{F}(v)}{2}\right],
 \]
 and
 \begin{align*}
  Z_{L(G)}(z^2)&=
  \sum_{F\subseteq E(G)}b^{\sum_{v\in V(G)}\binom{d_G(v)}{2}-\binom{d_F(v)}{2}-\binom{d_{G}(v)-d_{F}(v)}{2}}z^{\sum_{v}\deg_F(v)}\\
  &=b^{|E(L(G))|}\sum_{F\subseteq E(G)} \prod_{v\in V}b^{-\binom{d_G(v)-d_F(v)}{2}-\binom{d_F(v)}{2}}z^{d_F(v)}.
 \end{align*}

%
%
%
%

\end{proof}


We now observe that 
\[
  Z_{L(G)}(z)=b^{|E(L(G))| }Z_\mathcal{W}(G,1,(b^{-\binom{d_G(v)-i}{2}-\binom{i}{2}})_{i=0}^{d_G(v)};z^{1/2}).
\]
The relevant base polynomial is given by for a graph $G=(V,E)$,
\[
\Omega(G,1,x)=\prod_{e\in E} \left(1+\prod_{v\in e}x_v\right)
\]
and is clearly $\mathcal{S}[\pi/2]$-nonvanishing (and $\mathcal{D}$-nonvanishing).
Therefore, by Theorem~\ref{th:wagner}, to prove Theorem~\ref{th:main}, it is enough to show that the key polynomials have only real roots, as we will show in the next lemma.

For the remainder of this section, let us for fixed $b>1$, denote the relevant key polynomials by
\begin{equation}\label{eq:key ising}
  K_{d}(z)=\sum_{i=0}^d{\binom{d}{i}}b^{-\binom{d-i}{2}-\binom{i}{2}}z^i.
\end{equation}


\begin{lemma} \label{recursion}
 For any $d\ge 0$ we have
 \[
    K_{d+1}(z)=b^{-{d}}K_{d}(b z)+zK_{d}(b^{-1} z).
 \]
\end{lemma}
\begin{proof}
 By definition,
 \begin{align*}
  K_{d+1}(z)&=\sum_{i=0}^{d+1}\binom{d+1}{i}b^{-\binom{d+1-i}{2}-\binom{i}{2}}z^i\\
  &=\sum_{i=0}^{d+1}\left(\binom{d}{ i}+\binom{d}{i-1}\right)b^{-\binom{d+1-i}{2}-\binom{i}{2}}z^i\\
  &=\sum_{i=0}^{d}\binom{d}{i}b^{-\binom{d+1-i}{2}-\binom{i}{2}}z^i+\sum_{i=0}^{d}\binom{d}{i}b^{-\binom{d-i}{2}-\binom{i+1}{2}}z^{i+1}\\
  &=\sum_{i=0}^{d}\binom{d}{i}b^{-\binom{d-i}{2}-\binom{i}{2}}b^{-\binom{d-i}{1}}z^i + z\sum_{i=0}^{d}\binom{d}{i}b^{-\binom{d-i}{2}-\binom{i}{2}}b^{-\binom{i}{1}}z^{i}\\
  &=b^{-d}K_{d}(bz)+zK_{d}(b^{-1}z).
 \end{align*}
\end{proof}

\begin{lemma}\label{lemma:realroot}
 For any real number $b>1$ and integer $d\ge 1$ the key polynomial
\[
  K_d(z)=\sum_{i}{\binom{d}{i}b^{-\binom{d-i}{2}-\binom{i}{2}}}z^i
\]
has simple negative real roots, i.e., it is $\mathcal{S}[\pi]$-nonvanishing.
Moreover,  the largest root $z^{(d)}_0<0$ of $K_d(z)$ has absolute value at most $|z^{(d)}_0|<b^{-d}$.
\end{lemma}
\begin{proof}
 To prove the statement, we will show the following stronger statement: For any $d\ge 1$ the polynomial $K_d(z)$ has distinct zeros $0>z_1^{(d)}>\dots>z_{d}^{(d)}$ and for any $1\le i\le d-1$
 \[
    \frac{z^{(d)}_{i+1}}{z^{(d)}_{i}}>b^2.
 \]
 
 We will prove this statement by induction on $d$ using the identity of Lemma~\ref{recursion}, namely, 
 $K_{d+1}(z)=b^{-{d}}K_{d}(b z)+zK_{d}(b^{-1} z)$. If $d=1$, then the above statement on the zeros trivially holds. Assume that the statement is true for some $d\ge 1$. Let $P_1(z)=K_d(bz)$ and $P_2(z)=zK_d(b^{-1}z)$ and let $a_1>a_2>\dots>a_d$ be the zeros of $P_1(z)$ and let $0=b_1>b_2>\dots>b_{d+1}$ be the zeros of $P_2(z)$. 
 We refer to Figure~\ref{fig:zeros_key} for a schematic figure displaying the ideas of the proof.

    \begin{figure}[ht]
        \centering
        \begin{tikzpicture}[scale=1]
        \draw[->] (-7,0) -> (2,0);
        \filldraw(-6,0) circle (2pt) node[below] {$b_4$};
        \filldraw(-5.1,0) circle (2pt) node[above] {$z_3^{(3)}$};
        \filldraw(-4.35,0) circle (2pt) node[below] {$a_3$};
        \filldraw(-3.6,0) circle (2pt) node[below] {$b_3$};
        \filldraw(-3.0,0) circle (2pt) node[above] {$z_2^{(3)}$};
        \filldraw(-2.5,0) circle (2pt) node[below] {$a_2$};
        \filldraw(-1.6,0) circle (2pt) node[below] {$b_2$};
        \filldraw(-1.2,0) circle (2pt) node[above] {$z_1^{(3)}$};
        \filldraw(-0.8,0) circle (2pt) node[below] {$a_1$};
        \filldraw(0,0) circle (2pt) node[below] {$b_1$};
        \node[above] at (0,0) {$0$};
        
        \draw[decorate,decoration={brace,amplitude=4pt}] (-6,0.7) -- (-4.35,0.7) node [black,midway,yshift=+12pt] {$\cdot b^2$};
        \draw[decorate,decoration={brace,amplitude=4pt}] (-3.6,0.7) -- (-2.5,0.7) node [black,midway,yshift=+12pt] {$\cdot b^2$};
        \draw[decorate,decoration={brace,amplitude=4pt}] (-1.6,0.7) -- (-0.8,0.7) node [black,midway,yshift=+12pt] {$\cdot b^2$};
        
        \node at (-8,-1) {$\textrm{sign}(K_4(z))$};
        \node at (-6,-1) {$+$};
        \node at (-4.35,-1) {$-$};
        \node at (-3.6,-1) {$+$};
        \node at (-2.5,-1) {$+$};
        \node at (-1.6,-1) {$-$};
        \node at (-0.8,-1) {$-$};
        \node at (0,-1) {$+$};

        \end{tikzpicture}
    
        \caption{A schematic figure of the location of the zeros of $K_3(z)$,  $P_1(z)$ and $P_2(z)$.}
        \label{fig:zeros_key}
    \end{figure}
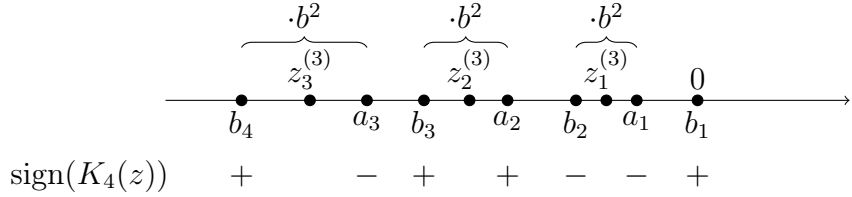

 We claim that $a_{i}>b_{i+1}$ and $b_{i}>a_{i}$ for any $0\le i\le d$. This is true indeed, since
 \[
    \frac{a_i}{b_{i+1}}=\frac{b^{-1}z_{i}^{(d)}}{bz_{i}^{(d)}}=b^{-2}<1
 \]
 and for $i\ge 1$
 \[
    \frac{b_i}{a_{i}}=\frac{bz_{i-1}^{(d)}}{b^{-1}z_{i}^{(d)}}<\frac{b^2}{b^2}=1.
 \]
 Therefore, the sequences $a_i$ and $b_i$ interlace, i.e. $b_{d+1}<a_d<b_d<\ldots<b_2<a_1<b_1=0$.
 Since the sign of $b^{-d}P_1(z)+P_2(z)=K_{d+1}(z)$ at the endpoints of the interval $[a_i,b_i]$ are different, therefore we know for sure that there is a zero of $K_{d+1}(z)$ in the interval $[a_i,b_i]$ for any $1\le i\le d$. To indicate a $d+1$th zero of $K_{d+1}(z)$, observe that $P_1(z)$ and $P_2(z)$ has different sign on $(-\infty, b_{d+1})$ and $P_2(z)$ has larger degree. 
 
 Since we found $d+1$ disjoint intervals containing at least one zero of $K_{d+1}(z)$, therefore each will contain exactly one zero, i.e., $K_{d+1}(z)$ has simple zeros $z_0^{(d+1)}>\dots>z_{d+1}^{(d+1)}$.
 
 To prove that the consecutive zeros are bounded, note that we have 
 \[
    \frac{z^{(d+1)}_{i+1}}{z^{(d+1)}_{i}}>\frac{b_{i+1}}{a_i}=\frac{bz_{i}^{(d)} }{b^{-1}z_{i}^{(d)}}=b^{2}.
 \]

Finally, note that $|z_{0}^{(d+1)}|<|a_1|=b^{-1}|z_{0}^{(d)}|$, which by induction is bounded by $b^{-d}z_0^{(1)}=b^{-d-1}$.
\end{proof}

\begin{remark}
 One could also easily obtain the real-rootedness of the polynomials $K_d(z)$ as follows. Let $P_{d}(z)=\sum_{i\ge 0}b^{-\binom{i}{2}}z^i$. Then it is well known (see e.g. \cite{brown}) that $P_d(z)$ is a real rooted polynomial, therefore $P^*_d(z)=z^{d}P_{d}(1/z)$ is also real rooted. Observe that the Schur-product of $P_d(z)$ and $P^*_d(z)$ is exactly $K_{d}(z)$, and since the Schur product preserves real-rootedness, we obtained that $K_d(z)$ has only real roots.
\end{remark}

Now we have all the ingredients to prove the Theorem~\ref{th:main}.
\begin{proof}[Proof of Theorem~\ref{th:main}]
 By Lemma~\ref{cl:express}, we can express $Z_{L(G)}(z,b)$ as a specialization of the multivariate subgraph counting polynomial, where each key polynomial coincides with one of the members of $\{K_{d}(z) ~|~ d\ge 1\}$. 
 By Lemma~\ref{lemma:realroot}, we know each key polynomial has to be real rooted, i.e., each $K^{(v)}(z)$ has no roots in $\{z\in\mathbb{C}~|~|\arg(z)|<\pi\}$. Choosing $\alpha=0$ in Theorem~\ref{th:wagner} we obtain the desired statement.
\end{proof}

Next we prove Theorem~\ref{th:fisher}.

\begin{proof}[Proof of Theorem~\ref{th:fisher}]
 Let $0<\alpha<\pi/2$ be fixed and take $C=\{z\in\mathbb{C}~|~|\arg(z)|\ge\pi-\alpha\}$. 
 Let us consider the following two variable polynomials
 \[
    K_{d}(z,t)=\sum_{i}\binom{d}{i}t^{\binom{d-i}{2}+\binom{i}{2}}z^i
 \]
 for some $1\le d\le \Delta$.
 
We know from Lemma~\ref{lemma:realroot} that for any fixed $b\geq 1$, all the roots of $K_d(z,b^{-1})$ are in $\textrm{int}(C)$. 
By continuity  we can find an open neighborhood $V_b$ of $b^{-1}$, such that $0\notin V_b$ and for any $1\le d\le \Delta$ and $\tau\in V_b$ the polynomial $K_d(z,\tau)$ has all its zeros contained in $\textrm{int}(C)$.
 Let $U:=\cup_{b\geq 1}V_b^{-1}$. Then $U$ is an open set containing $[1,\infty)$ and for any $b'\in U$ we have that the key polynomial $K_{d}(z,1/b')$ has no zeros contained in $\{z\in\mathbb{C}~|~|\arg(z)|<\pi-\alpha\}$.
 
Following the same argument as in the previous theorem, by the combination of Lemma~\ref{cl:express} and Theorem~\ref{th:wagner}, we obtain that for any $b'\in U$ the polynomial $Z_{L(G)}(\lambda,b')$ has no root in $\{z\in\mathbb{C}~|~|\arg(z)|<\pi-2\alpha\}$. 
In particular $Z_{L(G)}(1,b')\neq 0$ for all $b'\in U$.
\end{proof}

In a similar manner we can prove the following extensions:
\begin{proposition}
 For any graph $G$ and values $b_v\ge1$, a vector of variables $\lambda=(\lambda_v)_{v\in V}$, the multivariate polynomial
 \[
    F_{L(G)}(\lambda,(b_v)_{v\in V}):=\sum_{F\subseteq E}\prod_{v\in V}b_v^{-\binom{d_G(v)-d_F(v)}{2}-\binom{d_F(v))}{2}}\lambda_v^{d_F(v)}
 \]
 is weakly Hurwitz-stable, i.e. $\mathcal{S}[\pi/2]$-nonvanishing.
 
 Moreover, for any $\Delta\ge 2$ and $0<\alpha<\pi/2$ there exists a $[1,\infty)\subseteq U_{\Delta,\alpha}$, such that the following holds: for any graph $G$ of maximum degree $\Delta$ and any $b_v\in U_{\Delta,\alpha}$, the polynomial
 \[F_{L(G)}(\mathbf{\lambda},(b_v)_{v\in V})\neq 0\]    
 if $\lambda_v\in\{z\in\mathbb{C}\mid|\arg(z)|<\pi-2\alpha\}$.
\end{proposition}

\subsection{Disk around zero}
A priori, we know from \cite[Remark~24]{peters2020location}, that for any graph $H$ of maximum degree at most $\Delta\ge 2$ the polynomial $Z_H(\lambda)$ has no root in a disk around $0$ of radius for some $0<r<1/b$. This is the last ingredient to conclude Corollary~\ref{cor:FPTAS}.

For the sake of completeness and also as one more example for Wagner's theorem, we will describe  a zero-free disk around zero for every fixed $b>1$.
Let us denote by $\kappa_{d,b}$ the absolute value of the shortest zero of the key polynomial $K_d(z)$ from \eqref{eq:key ising}. Then:

\begin{proposition}
 Let $G$ be a graph of maximum degree at most $\Delta\geq 2$, then 
 \[
    Z_{L(G)}(z,b)\neq 0,
 \]
 if $|z|<\kappa^2_{\Delta,b}$.
\end{proposition}
\begin{proof}
From Lemma~\ref{lemma:realroot} we know that the polynomials $K_d(z)$ for $1\le d\le \Delta$ are non-zero in a disk of radius $\kappa_{\Delta,b}$. The polynomial $(1+xy)$ is $\mathcal{D}$-nonvanishing, thus the corresponding base polynomial is also $\mathcal{D}$-nonvanishing. Thus by the second part of Theorem~\ref{th:wagner}, if none of the key polynomials vanishes in a disk of radius $\kappa_{\Delta,b}$, then the corresponding subgraph counting polynomial 
\[
    Z_{L(G)}(z,b)=b^{|E(L(G))| }Z_\mathcal{W}(G,1,(b^{-\binom{d_G(v)-i}{2}-\binom{i}{2}})_{i=0}^{d_G(v)};z^{1/2})
\]
will be zero free in a disk of radius $\kappa_{\Delta,b}^2$.
\end{proof}

\section{Further remarks}

To finish the paper we would like to give two remarks.

One might wonder about the location of the zeros of the edge cover polynomials. First of all, if we look the closure of all possible zeros of $\mathcal{E}(\mathcal{H},z)$, then we would obtain the whole complex plane, since already the closure of the zeros of domination polynomials are dense in $\mathbb{C}$ (see \cite{brown2014roots}). But what happens, if we consider zeros of the edge cover polynomial of graphs? For instance, the closure of zeros of the edge cover polynomial of paths is dense in $[-4,0]$ (see \cite{csikvari2011roots} ), therefore the described region in Theorem~\ref{thm:edge cover strong} is tight for real zeros.


\begin{question}
 Is the set of zeros of edge-cover polynomials of graphs dense in the cardioid $\{-(1-\alpha)^2\mid |\alpha|\leq 1\}$?
\end{question}

\begin{figure}[ht]
\begin{subfigure}{.4\textwidth}
\centering
\begin{tikzpicture}
	\foreach \i in {1,2} {
	\node[vertex] (w\i) at (-\i, 0) {};
	\node[vertex] (e\i) at (\i,0) {};
	};
	\node[vertex] (u1) at (0,1) {};
	\node[vertex] (u2) at (0,-1) {};
	
	\draw (w2) -- (w1);
	\draw (w1) -- (u1);
	\draw (w1) -- (u2);
	\draw (e2) -- (e1);
	\draw (e1) -- (u1);
	\draw (e1) -- (u2);
	\draw (u1) -- (u2);
\end{tikzpicture}
\caption{A claw-free graph with $Z_{G}(\lambda,b)=\lambda^6 + (4 b^3 + 2 b) \lambda^5 + (b^6 + 11 b^4 + 3 b^2) \lambda^4 + (12 b^5 + 8 b^3) \lambda^3 + (b^6 + 11 b^4 + 3 b^2) \lambda^2 + (4 b^3 + 2 b) \lambda + 1$}
\end{subfigure}\qquad
\begin{subfigure}{.4\textwidth}
\centering
\begin{tikzpicture}[scale=1.3]
	\foreach \i in {0,1} {
        \node[vertex] (w\i) at (-1, \i) {};
        \node[vertex] (m\i) at (0,\i) {};
        \node[vertex] (e\i) at (1,\i) {};
	};
	
	\draw (w0) -- (w1);
	\draw (m0) -- (m1);
	\draw (e0) -- (e1);
    \draw (w0) -- (m0);
    \draw (w0) -- (m1);
    \draw (e0) -- (m0);
    \draw (e0) -- (m1);
    \draw (w1) -- (m0);
    \draw (w1) -- (m1);
    \draw (e1) -- (m0);
    \draw (e1) -- (m1);	
\end{tikzpicture}\vspace{1em}
\caption{A claw-free graph with $Z_{G}(\lambda,b)=\lambda^6 + (2 b^5 + 4 b^3) \lambda^5 + (b^8 + 12 b^6 + 2 b^4) \lambda^4 + (16 b^7 + 4 b^5) \lambda^3 + (b^8 + 12 b^6 + 2 b^4) \lambda^2 + (2 b^5 + 4 b^3) \lambda + 1$}
\end{subfigure}
\caption{Claw-free graphs where $Z^{MC}_G(\lambda)=\lambda^2(1+\lambda^2)$ is not real-rooted.}
\label{fig:clawfree}
\end{figure}

Our second remark concerns the Ising model. Another famous anti-ferromagnetic model  is the hard-core model, whose partition function is essentially the independence polynomial of graphs. It is well known that independence polynomials of line graphs have only real zeros \cite{heilmann1972theory}. Moreover, according to a theorem of Chudnovsky and Seymour \cite{chudnovsky2007roots}, the independence polynomial has only real zeros also for claw-free graphs. 
Thus, the natural question whether the anti-ferromagnetic Ising model $Z_G(\lambda,b)$ has only real zeros for claw-free graphs for any choice of $b\ge 1$ arises. 
The answer to this question is no.
To see why, let us fix a graph  $G$, such that $Z_G(\lambda)$ is real rooted for any $b>1$. Also, denote by $M$ the size of a largest cut in the graph $G$. Then as $b\to \infty$
\[
    \frac{Z_G(\lambda,b)}{b^M}\to \sum_{S\subseteq V}\lambda^{|S|}\mathbf{1}_{\textrm{$S$ realizes the maximum cut}}=:Z^{MC}_G(\lambda)
\]
has to be a real rooted polynomial. 
For the claw-free graphs in Figure~\ref{fig:clawfree} this is not case.




\end{document}